\documentclass[11pt]{article}

\usepackage{amsmath, amsthm, amssymb}
\usepackage{geometry}
\usepackage{tikz}
\usetikzlibrary{backgrounds,calc}
\usepackage{marginnote}
\usepackage{hyperref}

\newtheorem{thm}{Theorem}[section]
\newtheorem{lem}[thm]{Lemma}
\newtheorem{prop}[thm]{Proposition}
\newtheorem{cor}[thm]{Corollary}

\theoremstyle{definition}
\newtheorem{defn}[thm]{Definition}
\newtheorem{ex}[thm]{Example}
\newtheorem{obs}[thm]{Observation}

\newcommand{\RR}{\mathbb{R}}
\newcommand{\ZZ}{\mathbb{Z}}
\newcommand{\NN}{\mathbb{N}}
\newcommand{\CC}{\mathbb{C}}
\newcommand{\lift}[1]{\widetilde{#1}}

\DeclareMathOperator{\rank}{rank}
\DeclareMathOperator{\rankzn}{rank_{\ZZ}^+}
\DeclareMathOperator{\rankp}{rank^{+}}
\DeclareMathOperator{\spa}{span}

\title{A Yannakakis-type theorem for lifts of affine semigroups}

\author{
	Jo\~ao Gouveia\thanks{CMUC, Department of Mathematics, University of Coimbra, 3001-454 Coimbra, Portugal.
		This author was partially supported by Centro de Matemática da Universidade de Coimbra (CMUC),
 funded by the Portuguese Government through FCT/MCTES, DOI 10.54499/UIDB/00324/2020.
		Email: \href{jgouveia@mat.uc.pt }{jgouveia@mat.uc.pt}.} \and
	Amy Wiebe\thanks{Department of Mathematics, University of British Columbia, Kelowna, British Columbia, V1V 1V7, Canada.
		This author was partially funded by the Natural Sciences and Engineering Research Council of Canada (cette recherche est partiellement financ\'ee par le Conseil de recherches en sciences naturelles et en g\'enie du Canada), Discovery  Grant \#2024-04643.
		Email: \href{amy.wiebe@ubc.ca}{amy.wiebe@ubc.ca}.}
}

\date{}

\begin{document}

\maketitle

\begin{abstract}
    Yannakakis' theorem relating the extension complexity of a polytope to the size of a nonnegative factorization of its slack matrix is a seminal result in the study of lifts of convex sets. Inspired by this result and the importance of lifts in the setting of integer programming, we show that a similar result holds for the discrete analog of convex polyhedral cones---affine semigroups. We define the notions of the integer slack matrix and a lift of an affine semigroup. We show that many of the characterizations of the slack matrix in the convex cone setting have analogous results in the affine semigroup setting. We also show how slack matrices of affine semigroups can be used to obtain new results in the study of nonnegative integer rank of nonnegative integer matrices. 
\end{abstract}

\section{Introduction}

A considerable body of work exists studying the possible approaches for dealing with integer linear programs. A commonly used procedure is to replace the discrete feasible set by its convex hull and try to find a suitable compact description of that hull or, at least, a relaxation of it. In order to find such a description, we often resorts to lifts; that is, we define a higher dimensional set that projects onto the set we are interested in and that has, in some sense, a small description. This general framework is popular both in theory and in practice and has provided deep insights into several important problems. 

Conceptually, one can think of this approach as a two-step process: first we go from the discrete setting to the convex setting, and there we use a lifting procedure. The motivation for the present paper stems from a simple question,  apparently unaddressed by the existing body of literature: what happens if we try to bypass one of the steps and immediately define a lift that respects the structure of the integer points within the original feasible region? In this work we attempt to set up the machinery needed to address this question, and take the first steps in establishing a theory of discrete lifts, obtaining some interesting byproducts in the process.

The natural objects to focus on when attempting such a task are affine semigroups. They are the discrete analog of convex polyhedral cones, and play an important role in the theory of integer programming. Thus, it makes sense to begin by exploring how the tools used in the study of convex cones may be expanded to this discrete setting. Moreover, affine semigroups appear as objects of importance in a variety of other areas including algebraic geometry, commutative algebra, number theory, and combinatorics. This suggests that any results obtained will have interesting interpretations in other domains of mathematics. 

The central result in  the study of lifts of convex sets is Yannakakis' theorem \cite{Yann} and its variants. Associated to any polyhedral convex cone there exists a nonnegative matrix, called its slack matrix, that consists of the evaluations of a set of representatives of the extreme rays of the dual cone at a set of representatives of the extreme rays of the cone itself. In its polyhedral convex cone version, Yannakakis' theorem states that we can write a polyhedral cone as a projection of a polyhedral cone with $k$ facets if and only if we can factorize its slack matrix as a product of two nonnegative matrices with the inner dimension of the product being $k$.

In order to get a discrete version of this result we must start by defining a notion of slack matrix for affine semigroups. Slack matrices of polyhedral cones are an interesting class of nonnegative matrices in their own right. They have been studied in detail, for example in \cite{GOUVEIA20132921}, and offer canonical representations for cones that can be used for the study of realization spaces of polytopes \cite{gouveia2019slack}. We show that the integer slack matrices we define satisfy many of the characterizations of general slack matrices of cones, and also present a canonical representation of the semigroup. Armed with this new notion, we proceed to introduce a definition for a lift of an affine semigroup. Again, we are inspired by the notion of a lift for polyhedral cones. 

Our main result is that a Yannakakis-type theorem also holds in the affine semigroup setting. More precisely, the minimal size of a lift of a semigroup equals the nonnegative integer rank of its slack matrix. We briefly explore the potential of interpreting this result in the settings of integer programming and toric geometry. 
The notion of nonnegative integer rank is not new, and there have been some limited previous results from a matrix theory point of view. We end the paper by exploiting the structure of slack matrices of $2$-dimensional semigroups to get new results in the study of nonnegative integer rank of a matrix.

The remainder of this paper is organized as follows. In Section~\ref{sec:bg}, we provide some background and notation for working with affine semigroups. 
In Section~\ref{sec:slack}, we introduce the notion of the slack matrix of an affine semigroup and give some characterizations of these matrices following~\cite{GOUVEIA20132921}. In Section~\ref{sec:lift}, we define the lift of an affine semigroup and present a Yannakakis-type theorem relating lifts of affine semigroups to factorizations of their slack matrices. We also include some discussion on the application of these lifts to toric geometry and integer programming. Finally, in Section~\ref{sec:nn_int_rank}, we use slack matrices of 2-dimensional affine semigroups to provide new results on the nonnegative integer rank of nonnegative integer matrices. Section~\ref{sec:concl} concludes the work with some remaining questions.

\section{Background and definitions}
\label{sec:bg}

A \emph{lattice} in $\RR^n$ is a discrete additive subgroup of $\RR^n$, i.e., a set $\Lambda \in \RR^n$ such that $x-y \in \Lambda$ for all $x,y \in \Lambda$ and every element of $\Lambda$ is isolated. The most important example of such a lattice is $\ZZ^n$. A \emph{basis} of a lattice is a set of linearly independent vectors $\{v_1,...,v_d\} \in \Lambda$ such that 
$$\Lambda=\{ m_1 v_1 + m_2 v_2 + \cdots + m_d v_d \, | \, m_1,m_2,...,m_d \in \ZZ\}.$$
The \emph{dual lattice} to $\Lambda$ is the lattice 
$$\Lambda^*=\{w \in \textup{span}(\Lambda) \, | \, \langle w,v \rangle \in \ZZ, \textup{ for all } v \in \Lambda\}.$$
\begin{ex} \label{ex:latticedual}
Consider the lattice $\Lambda=\{ (x,y) \in \ZZ^2 \, | \, x+y \in 2 \ZZ\}$. A basis for this lattice is, for instance, $B=\{(2,0),(1,-1)\}$. Its dual is the lattice
$$\Lambda^*=\{(a,b) \in \RR^2 \, : \, ax + by \in \ZZ, \textup{ for all } (x,y) \in \Lambda\}=\{(a,b) \in (\ZZ/2)^2 \, | \, a+b \in \ZZ\}.$$
\begin{figure}
\begin{center}
\includegraphics[width=0.25\linewidth]{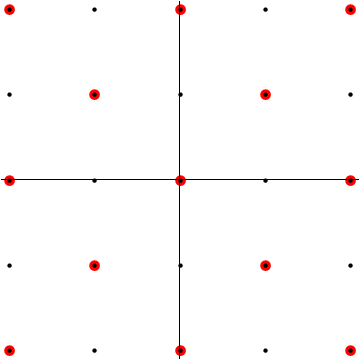}
\hspace{2cm}
\includegraphics[width=0.25\linewidth]{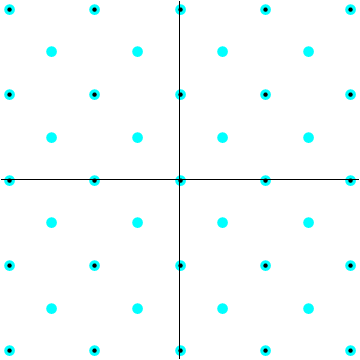}
\end{center}
\caption{Lattice of Example~\ref{ex:latticedual} and its dual.} \label{fig:lattice_duality}
\end{figure}
In Figure \ref{fig:lattice_duality} we can see both primal and dual lattices.
\end{ex}

We are interested in \emph{affine semigroups}. We can think of these as sets $\Gamma \in \RR^n$ that are closed under addition, contain $0$, and are finitely generated with the additional property that the set $\Lambda_{\Gamma}=\{x-y \, | \, x,y \in \Gamma\}$ is a lattice. In other words, sets of the form
$$    \Gamma=\{  m_1 v_1 + m_2 v_2 + \cdots + m_k v_k \, | \, m_1,m_2,...,m_k \in \ZZ_+ \}$$
for some $v_1,...,v_k \in \RR^n$, not necessarily linearly independent, such that 
$$\Lambda_{\Gamma}=\{  m_1 v_1 + m_2 v_2 + \cdots + m_k v_k \, | \, m_1,m_2,...,m_k \in \ZZ \}$$
is a lattice, i.e., it has no accumulation points.

Associated to any affine semigroup $\Gamma$, there is the cone of all real nonnegative combinations of elements of $\Gamma$. In other words,
$$K_{\Gamma}=\{ \lambda_1 v_1 + \lambda_2 v_2 + \cdots + \lambda_k v_k \, | \,\lambda_1,\lambda_2,...,\lambda_k \in \RR_+ \}  $$
where $v_1,...,v_k$ span $\Gamma$. Note that $\Gamma \subseteq \Lambda_{\Gamma} \cap K_{\Gamma}$. Moreover, Gordan's Lemma guarantees that 
$\Lambda_{\Gamma} \cap K_{\Gamma}$ is finitely generated, and hence is also an affine semigroup. We  call this semigroup the \emph{normalization} of $\Gamma$ and denote it by $\overline{\Gamma}$. An affine semigroup $\Gamma$ is said to be \emph{normal} if $\Gamma=\overline{\Gamma}$,

\begin{ex} \label{ex:affine}
Consider the affine semigroup $\Gamma$ generated by $\{(1,0),(1,1),(3/2,3/2)\}$. Then, $\Lambda_{\Gamma}=\{(a,b) \in (\ZZ/2)^2 \, | \, a+b \in \ZZ\}$, the same as the dual lattice in Example \ref{ex:latticedual}. Moreover $K_{\Gamma}$ is the cone spanned by $(1,0)$ and $(1,1)$. In Figure \ref{fig:affsemigroup} we can see the semigroup, its lattice and its cone.
\begin{figure}
\begin{center}
\includegraphics[width=0.25\linewidth]{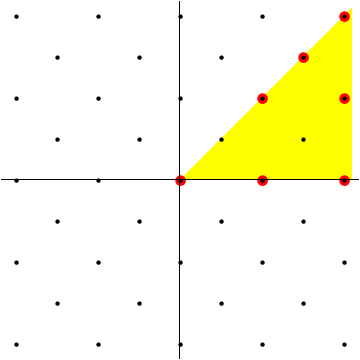}
\end{center}
\caption{Affine semigroup of Example~\ref{ex:affine} with its associated lattice and cone.} \label{fig:affsemigroup}
\end{figure}
We can see that $\Gamma \not = \overline{\Gamma}$, since there are points in $K_{\Gamma} \cap \Lambda_{\Gamma}$ that are not in $\Gamma$,  hence this affine semigroup is not normal.
   
\end{ex}

Affine semigroups are a discrete counterpart to convex cones, which have a rich theory of duality. We wish to explore how this duality translates to the discrete setting. The dual of a convex cone $K \subseteq \RR^n$ is the set $K^*=\{ y \in \RR^n \, | \, \langle x,y \rangle \geq 0 \textrm{ for all } x \in K\}$, which we can identify as the set of linear operators wich take nonnegative values on $K$, i.e, with the homomorphism cone $\textup{Hom}(K,\RR_+)$.

A simple idea to generalize this is then to define the dual of an affine semigroup $\Gamma$ as 
$$\Gamma^*=\Lambda_{\Gamma}^* \cap K_{\Gamma}^*,$$ which can be identified as $\textup{Hom}(\Gamma,\ZZ_+)$. This is the definition seen, for instance, in \cite{ROSALES1999175}. 

However, this definition is not always straightforward. 

\begin{ex}
Consider $\Gamma=\ZZ^n$. Then $\Lambda_{\Gamma}=\Lambda_{\Gamma}^*=\Gamma$ and $K_{\Gamma}=\RR^n$ which implies $K_{\Gamma}^*=\{0\}$, so $\Gamma^*=\{0\}$. By our definitions $\Lambda_{\Gamma^*}=\Gamma^*=\{0\}$, and therefore $(\Gamma^*)^*=\{0\}$.
\end{ex}

The main issue here is that we would like both primal and dual cones to be full dimensional. This ensures that the intersection of the lattice with the cone contains a set of generators for the lattice. The primal cone ambient space is the span of the lattice, so it is always full dimensional. In order to guarantee that the dual cone is also full dimensional we need to enforce that the primal cone is pointed, i.e, that it contains no lines. This is the same as enforcing that the affine semigroup $\Gamma_p$  is  \emph{positive}, i.e., zero is the only element $x \in \Gamma_p$ such that $-x \in \Gamma_p$.

\begin{lem} \label{lem:semigroup_dual_properties}
    Given a positive affine semigroup $\Gamma\subset\RR^n$, we have:
    \begin{enumerate} 
    \item $K_{\Gamma^*}=K_{\Gamma}^*$ \label{lem:pos_dual_cone}
    \item $\Lambda_{\Gamma^*} = \Lambda_\Gamma^*$;
    \label{lem:pos_dual_lattice}
    \item $(\Gamma^*)^*=\overline{\Gamma}$;
    \label{lem:pos_dual_semi}
    \item If $\Gamma$ is normal, $(\Gamma^*)^*=\Gamma$.
    \label{lem:pos_dual_normal}
    \end{enumerate}
\end{lem}
\begin{proof}
For \eqref{lem:pos_dual_cone} the first thing we have to note is that $K_{\Gamma}$ always has the same dimension as the span of $\Lambda_\Gamma$, since by definition it contains a basis of that space. Furthermore, it is pointed by the positivity of $\Gamma$, hence the basic properties of convex cones tell us that $K_{\Gamma}^*$ is also a pointed full-dimensional cone. Note that ${\Gamma^*} \subseteq K_{\Gamma}^*$ by definition so we have $K_{\Gamma^*} \subseteq K_{\Gamma}^*$ . Furthermore,  the dual of a rational cone is always rational itself, so there is a collection of points in $\Lambda_\Gamma^*$ that generate the cone $K_{\Gamma}^*$. By definition these points belong to $\Gamma^*$, and so we have the inclusion $K_{\Gamma}^* \subseteq K_{\Gamma^*}$.

For \eqref{lem:pos_dual_lattice}, since $K_{\Gamma}^*$ has an open interior it contains arbitrarily large open balls centered around lattice points. That means that given some generating set for $\Lambda_\Gamma^*$ we can find a lattice translate of it that is entirely inside $K_{\Gamma}^*$, which means it is part of 
$\Gamma^* $. Hence $\Lambda_\Gamma^*$ is contained in $\Lambda_{\Gamma^*}$ and since the other inclusion is direct from definition we have the second statement.

From these first two results, \eqref{lem:pos_dual_semi} follows immediately since $(\Gamma^*)^*=(K_{\Gamma^*})^* \cap (\Lambda_{\Gamma^*})^*$. 
Then by \eqref{lem:pos_dual_cone}, $(K_{\Gamma^*})^* = K_{\Gamma}$, and by \eqref{lem:pos_dual_lattice}, $(\Lambda_{\Gamma^*})^* = (\Lambda_{\Gamma}^*)^*= \Lambda_{\Gamma}$, where the last equality follows from basic lattice theory. 

The last statement follows immediately from \eqref{lem:pos_dual_semi}.
\end{proof}

Lemma~\ref{lem:semigroup_dual_properties} shows that when restricted to positive normal affine semigroups, this definition of the dual affine semigroup is nicely symmetric. However, if we consider non-normal affine semigroups, we lose symmetry: while every affine semigroup has a dual, not every affine semigroup is a dual of another affine semigroup. One could sacrifice the unicity of the dual to fix this shortcoming but in this paper we will content ourselves to lose symmetry and work with this more straightforward definition. 

\begin{ex} \label{ex:affine_dual}
Let us revisit the affine semigroup of Example \ref{ex:affine}. The dual lattice to $\Lambda_{\Gamma}$ is $\Lambda_\Gamma^*=\{ (x,y) \in \ZZ^2 \, | \, x+y \in 2 \ZZ\}$, as we saw in Example \ref{ex:latticedual}. The dual cone to $K_{\Gamma}$ is the cone $K_\Gamma^*$ that is spanned by $(0,1)$ and $(1,-1)$. Now $\Gamma^*$ is the normal affine semigroup obtained by intersecting $\Lambda_\Gamma^*$ and $K_\Gamma^*$, i.e., the normal semigroup generated by $(0,2)$ and $(1,-1)$, represented in Figure \ref{fig:affsemigroupdual}. 
\begin{figure}
\begin{center}
\includegraphics[width=0.25\linewidth]{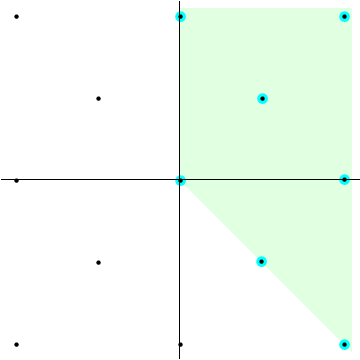}
\end{center}
\caption{Affine semigroup of Example~\ref{ex:affsemigroupdual} its associated lattice and cone.} \label{fig:affsemigroupdual}
\end{figure}
\label{ex:affsemigroupdual}
\end{ex}

\section{Slack matrix of a positive affine semigroup}
\label{sec:slack}

In what follows, we will restrict ourselves to positive affine semigroups, so that the conclusions of Lemma \ref{lem:semigroup_dual_properties} always hold.  Another nice consequence of enforcing positivity is the following fact, found for instance in \cite{Bruns2001}.

\begin{lem}
A positive affine semigroup has a unique minimal generating set.
\label{lem:pos_mingenset}
\end{lem}

This means that we can associate to any positive affine semigroup $\Gamma$ a unique (up to row and column permutations) nonnegative integer matrix in the following way. 

\begin{defn}
Let $\Gamma$ be a positive affine semigroup. Let $a_1,...,a_k$ and $b_1,...,b_l$ be the sets of minimal generators of $\Gamma$  and $\Gamma^*$, respectively. We define the \emph{slack matrix} of $\Gamma$ as the nonnegative integer matrix 
$S_{\Gamma} \in \ZZ_+^{k \times l}$ whose $(i,j)$-entry is $\langle a_i,b_j \rangle$.
\end{defn}

Note the generators of a positive affine semigroup $\Gamma$ must include generators of the rays of the cone $K_{\Gamma}$. Thus the slack matrix $S_{\Gamma}$ is a slack matrix of $K_{\Gamma}$ in the sense of slack matrices of convex cones (see \cite{GOUVEIA20132921}), possibly with redundant rows and columns.

\begin{ex} \label{ex:slackmatrixex} Let us revisit the affine group  $\Gamma$ generated by $\{(1,0),(1,1),(3/2,3/2)\}$, introduced in Example \ref{ex:affine}. It is simple to check this is indeed the minimal set of generators. Furthermore, by Example \ref{ex:affine_dual} we know that its dual is generated by 
$(0,2)$ and $(1,-1)$. So by the above definition its slack matrix is the $3 \times 2$ nonnegative integer matrix
$$S_{\Gamma}=\begin{bmatrix}
    0 & 1 \\
    2 & 0 \\
    3 & 0
\end{bmatrix}.$$
\end{ex}

By Lemma~\ref{lem:pos_mingenset}, the slack matrix of a positive affine semigroup is unique up to permutation of rows and columns. In what follows, we explore the deeper connections between an affine semigroup and its slack matrix. 

The first important fact is that the slack matrix provides a canonical embedding of an affine semigroup in the integer lattice. To state this more precisely we introduce two affine semigroups that can be associated to any integer matrix. Let $M \in \ZZ^{n \times m}$. Denote by $\Gamma_M^{\textup{row}} \subseteq \ZZ^m$ and $\Gamma_M^{\textup{col}} \subseteq \ZZ^m$, the semigroups spanned by all nonnegative integer combinations of the rows and columns of $M$, respectively.

\begin{lem} \label{lem:isomorphismslack}
Let $\Gamma$ be a positive affine semigroup, and let $S \in \ZZ_+^{k \times l}$ be its slack matrix. Then $\Gamma \cong \Gamma_S^{\textup{row}}$ and $\Gamma^* \cong \Gamma_S^{\textup{col}}$, where the isomorphisms are given by sending each generator of $\Gamma$ or $\Gamma^*$ to the corresponding  row or column, respectively, of the slack matrix.
\end{lem}
\begin{proof}
 Let $a_1,...,a_k$ and $b_1,...,b_l$ be the minimal generating sets for $\Gamma$ and $\Gamma^*$, respectively. Let $\varphi:\Gamma \rightarrow \ZZ^l$ be the map defined by
$$\varphi\left(\sum n_i a_i\right) = \sum n_i S_i$$
for all $n_1,...,n_k \in \ZZ_+$, where $S_i$ is the $i$-th row of $S$. We want to show that $\varphi$ is well-defined and injective.

We begin by showing that this map is well-defined. If $\sum n_i a_i = \sum m_i a_i$, then for $j=1,...,l$ we have
$$\left[\sum (n_i-m_i)S_i\right]_j= \sum (n_i-m_i)S_{ij}= \sum (n_i-m_i) \langle a_i, b_j \rangle = 
\langle \sum (n_i-m_i) a_i ,b_j\rangle =0.$$
Since the set of $b_j$ span the space, this implies $\sum n_i S_i=\sum m_i S_i$.

To show $\varphi$ is injective, we will show that for any integer vector $k$ such that $k^\top S=\sum k_i S_i =0$, we have $\sum k_i a_i=0$. If $k^\top S=0$ then, for each $j=1,...,l$, we must have $\langle \sum k_i a_i, b_j \rangle = 0$. This means that both $\sum k_i a_i$ and $-\sum k_i a_i$ are in $K_{\Gamma}$ as this is the dual of the cone spanned by the $b_j$. Since ${\Gamma}$ is positive, this cone is pointed, so we must have $\sum k_i a_i=0$ as intended. This concludes the proof of $\Gamma \cong \Gamma_S^{\textup{row}}$. 

The proof that  $\Gamma^* \cong \Gamma_S^{\textup{col}}$ is analogous with the roles of $a_i$ and $b_j$ exchanged.
\end{proof}

\begin{ex} Consider the affine semigroup $\Gamma$ of Example \ref{ex:slackmatrixex}, which is generated by $\{(1,0),(1,1),(3/2,3/2)\}$. Looking at its slack matrix, Lemma \ref{lem:isomorphismslack} states that $\Gamma$ is isomorphic to the affine semigroup generated by $\{(0,1),(2,0),(3,0)\}$. In fact we can see that this is simply the image of $\Gamma$ by the lattice isomorphism $(x,y) \rightarrow (x+y,x-y)$. 

Moreover, we see that $\Gamma^*$, which is generated by $\{(0,2),(1,-1)\}$, is isomorphic to the affine semigroup in $\ZZ^3$ generated by $\{(0,2,3), (1,0,0)\}$. The lattice map which sends $(x,y)$ to the integer point $(x,x+y,3(x+y)/2)$ gives an isomorphism between $\Gamma^*$ (see Example~\ref{ex:affsemigroupdual}) and $\ZZ_+^3$ intersected with the plane cutout by $3y=2z$.

\end{ex}

Lemma \ref{lem:isomorphismslack} shows that the slack matrix can be seen as a canonical embedding of the equivalence class of the affine semigroup up to isomorphism. In particular, it is clear that two affine semigroups have the same slack matrix if and only if they are isomorphic.

In fact, we can say more about this canonical embedding. The result of \cite[Theorem 1]{GOUVEIA20132921}, for slack matrices of convex cones in general, translates to
$$K_S^{\textup{row}}=\RR^l_+ \cap \textup{Row}(S),$$
 where $\textup{Row}(S)$ is the row space of $S$, and $K_S^{\textup{row}}$ is the cone generated by its rows.  Similarly, by symmetry, 
$$K_S^{\textup{col}}=\RR^k_+ \cap \textup{Col}(S),$$
  where $\textup{Col}(S)$ is the column space of $S$ and $K_S^{\textup{col}}$ is the cone generated by its columns. To these results about cones, we can add analogous results about lattices. To do this,  as we did with affine semigroups, we associate to an integer matrix $M \in \ZZ^{n \times m}$ two lattices:  $\Lambda_M^{\textup{row}} \subseteq \ZZ^m$ and $\Lambda_M^{\textup{col}} \subseteq \ZZ^n$, the lattices spanned by all integer combinations of the rows and columns of $M$, respectively.

  \begin{lem} \label{lem:LGC}
  Let $\Gamma$ be a positive affine semigroup with slack matrix $S\in\ZZ_+^{k\times l}$. Then
  \begin{enumerate}
      \item $\Lambda_S^{\textup{row}}=\ZZ^l \cap \textup{Row}(S),$
      \item  $\Lambda_S^{\textup{col}}=\ZZ^k \cap \textup{Col}(S).$
  \end{enumerate}
  \end{lem}

  \begin{proof}
  We prove the statement for rows, and note that the proof for columns, (2), is analogous. 
  
  The inclusion $\Lambda_S^{\textup{row}} \subseteq \ZZ^l \cap \textup{Row}(S),$ is clear so we just need to prove the reverse. Let $x \in \ZZ^l \cap \textup{Row}(S)$. Then there exist real $\lambda_i$ such that
  $x=\sum_{i=1}^k \lambda_i S_i$, where $S_i$ is row $i$ of~$S$. This means that for every $j=1,\ldots,l$ we have
  $$x_j = \sum_{i=1}^k \lambda_i \langle a_i,b_j \rangle  = \langle \sum_{i=1}^k \lambda_i a_i, b_j \rangle$$
  where the $\{a_1,\ldots, a_k\}$ and $\{b_1,\ldots, b_l\}$ are the minimal generators of the affine semigroup~$\Gamma$ and its dual $\Gamma^*$, respectively. Since each $x_j\in\ZZ$, this
  means that the linear operator $\sum \lambda_i a_i$ takes integer values on every element of $\Gamma^*$;  therefore, it belongs to $\Lambda_{\Gamma^*}^*=\Lambda_{\Gamma}$ (see Lemma~\ref{lem:semigroup_dual_properties}) so it must be possible to write it as an integer combination of the generators. Hence we can assume $\lambda \in \ZZ^l$, and so $x \in \Lambda_S^{\textup{row}}$ as intended.
  \end{proof}

It turns out that these conditions offer a characterization of slack matrices.

\begin{thm} \label{thm:classification}
Let $S$ be a nonnegative $k \times l$ integer matrix. Then $S$ is the slack matrix of a positive affine semigroup if and only if the following conditions hold.
\begin{enumerate}
\item $K^{\textup{row}}_S=\RR^l_+ \cap \textup{Row}(S)$,  $\Lambda_S^{\textup{row}}=\ZZ^l \cap \textup{Row}(S)$ and the rows of $S$ are a minimal generating set for $\Gamma_S^{\textup{row}}$;
\item The columns of $S$ are a  minimal generating set for $\ZZ^k_+ \cap \textup{Col}(S)$.
\end{enumerate}
\end{thm}
\begin{proof}

When $S$ is a slack matrix, the conditions in (1) hold by previous results: for the cone by \cite[Theorem 1]{GOUVEIA20132921}; for the lattice by Lemma \ref{lem:LGC}; and for the rows being a minimal generating set using the isomorphism of Lemma~\ref{lem:isomorphismslack}.
For (2), note that the same conditions proven for the rows in (1) also were proven for the columns. Note then that $\Gamma_S^{\textup{col}} \cong \Gamma^*$, hence it is normal. This means that it equals $K^{\textup{col}}_S \cap \Gamma^{\textup{col}}_S = \ZZ^k_+ \cap \textup{Col}(S)$. 

We will now suppose now that $S$ verifies the conditions and show that it is indeed a slack matrix.
Suppose $S$ has rank $d$ and let $S=V^\top W$ be an integer factorization through $\ZZ^d$. We will show that $S$ is the slack matrix of the positive affine semigroup $\Gamma_V^{\textup{col}}$. 

Since the rows and columns of $S$ are minimal generating sets for $\Gamma_S^{\textup{row}}$ and $\Gamma_S^{\textup{col}}$, respectively, both $\Gamma_V^{\textup{col}}$ and $\Gamma_W^{\textup{col}}$ must be minimally spanned by the columns of $V$ and $W$, respectively. To see this, notice that if there is a column of $V$ (or $W$) that can be written as a nonnegative combination of its other columns, then the corresponding row (or column) of $S$ can also be written in the same way. Therefore,  it is enough to show that $(\Gamma_V^{\textup{col}})^*= \Gamma_W^{\textup{col}}$.

We first need  to prove that the cones spanned by the columns of $W$ and $V$ are dual. This is immediate from the general theory of slack matrices of cones, but we include the argument here for completeness. It is clear that $K_W^{\textup{col}} \subseteq (K_V^{\textup{col}})^*$ since all inner products between generators are nonnegative. It remains to show the reverse inclusion. If $w \in (K_V^{\textup{col}})^*$, then  we have $V^\top w \in \RR^k_+$. Furthermore, $V^\top w$ is in the row space of $V$ which is the same as the column space of $S$. But then our first condition implies $V^\top w \in K_S^{\textup{col}}$, so $V^\top w=V^\top W \lambda$ for $\lambda \geq 0$. This implies $w=W \lambda$, since $V$ has rank $d$, hence we conclude $w \in K_W^{\textup{col}}$ as intended.

Now we need to show that $\Lambda_V^{\textup{col}}$ and $\Lambda_W^{\textup{col}}$ are also dual to each other. We prove this by showing that  $\Lambda_V^{\textup{col}} = \Lambda_W^{\textup{col}} = \ZZ^d$. 
Take any $y \in \ZZ^d$ and note that $V^\top y \in \ZZ^k \cap \textup{Col}(S)$, which means by our condition that $V^\top y \in \Lambda_S^{\textup{col}}$, which means there exist $z \in \ZZ^l$ such that $V^\top y=V^\top Wz$. Since $V$ has rank $d$, this implies $y=Wz$ and so $y \in \Lambda_W^{\textup{col}}$, as intended. Similarly we can show   $\Lambda_V^{\textup{col}} = \ZZ^d$ so we have the intended result.

We have shown $\Gamma_W^{\textup{col}} \subseteq(\Gamma_V^{\textup{col}})^* = \ZZ^d \cap K_W^{\textup{col}}$. 
To finish the proof just note that if $w \in  \ZZ^d \cap K_W^{\textup{col}}$ then $V^\top w \in \ZZ^k_+ \cap \textup{Col}(S)$ hence, by our conditions, it must be a nonnegative combination of the columns of $S$, i.e., there are $\lambda \in \ZZ_+^l$ such that $V^\top w=V^\top W \lambda$. Since $w \in \textup{Col}(W)$, this implies $w=W \lambda$,
so $w \in \Gamma_W^{\textup{col}}$, completing the proof.
\end{proof}

\begin{obs}
    The condition on the cone $K^{\textup{row}}_S$ in Theorem \ref{thm:classification} (1) is what in \cite{GOUVEIA20132921} is called the \emph{row cone generating condition} (RCGC). The equivalent condition for columns that is implied by (2) is known as the \emph{column cone generating condition} (CCGC). They are in fact equivalent.

    By analogy, one could call the condition on the lattice in (1) the \emph{row lattice generating condition} (RLGC) and the equivalent condition for columns, implied by (2), the \emph{column lattice generating condition} (CLGC). 
\end{obs}

It is possible to rewrite Theorem \ref{thm:classification} in a slightly more succinct way.

\begin{cor}
Let $S$ be a nonnegative $k \times l$ integer matrix. Then, after erasing any row or column that is a nonnegative combination of the others, $S$ is the slack matrix of a positive affine semigroup if and only if $\overline{\Gamma_S^{\textup{row}}}=\ZZ^l_+ \cap \textup{Row}(S)$ and  ${\Gamma_S^{\textup{col}}}=\ZZ^k_+ \cap \textup{Col}(S)$.
\end{cor}

\section{Lifts of affine semigroups}
\label{sec:lift}

One of the main motivations for the study of slack matrices for general polyhedral cones is to study their lifts, i.e., ways of representing them as projections of higher dimensional polyhedral cones. It is natural to ask, then, whether this connection extends to affine semigroups---the discrete analog for polyhedral cones---and their slack matrices.

\begin{defn} \label{def:lift}
Let $\Gamma$ be a positive affine semigroup. A {\em lift} of $\Gamma$ is another positive affine semigroup $\lift{\Gamma}$ together with an affine semigroup homomorphism $\iota : \Gamma \rightarrow \lift{\Gamma}$ such that the adjoint map verifies $\iota^*(\lift{\Gamma}^*)=\Gamma^*$. The {\em size} of the lift is the size of the minimal generating set of $\lift{\Gamma}$. 
\end{defn}

Note that the name \emph{lift}, given by analogy to the convex cone case, has a slightly different interpretation in the case of affine semigroups. In convex cones, a lift of a cone $K$ is a cone $\lift{K}$ that linearly projects to it. This implies that the adjoint sends the dual cone $K^*$ into the dual cone of the lift $\lift{K}^*$. In the convex cone case, because the role of the primal and dual cones are interchangeable, this is not very important, however, in the affine semigroup case we lose that symmetry, and the role of $\Gamma$ is analogous to that of $K^*$ in the above description. We are taking a somewhat dual view of what a lift is in this paper, but keeping the name.

It is not hard to check that lifts are always injective. To see this, just note that for any two distinct elements $x$, $y$ of $\Gamma$ there is an element of the dual that gives them different evaluations. This element will be the pullback of some element in the dual of $\lift{\Gamma}$ that, by definition, will have different values at $\iota(x)$ and $\iota(y)$ so they must also be distinct. The simplest lifts one can make, in fact, are mere inclusions.

\begin{ex}
Consider the affine semigroup $\Gamma$ generated by $(1,0), (1,1)$ and $(0,2)$. This semigroup is all of $\ZZ^2_+$ except the elements $(0,k)$ for $k$ odd. If we take $\lift{\Gamma}=\ZZ^2_+$, the inclusion map from $\Gamma$ to $\lift{\Gamma}$ is a lift. This is clear since $\Gamma^*=\lift{\Gamma}^*$, and the pullback map is simply the identity.
\end{ex}

It is possible, however, to define more interesting lifts. As in the convex case, one can increase the dimension to find smaller lifts.

\begin{ex}\label{ex:lift}
Consider the normal affine semigroup $$\Gamma=\{(x,y) \in \ZZ^2 \, | \, 0 \leq y \leq \frac{12}{5} x  \}.$$
It has five generators and is pictured with its dual in Figure \ref{fig:slackmatrixsemigroups}. 

\begin{figure}
\includegraphics[width=0.35\linewidth]{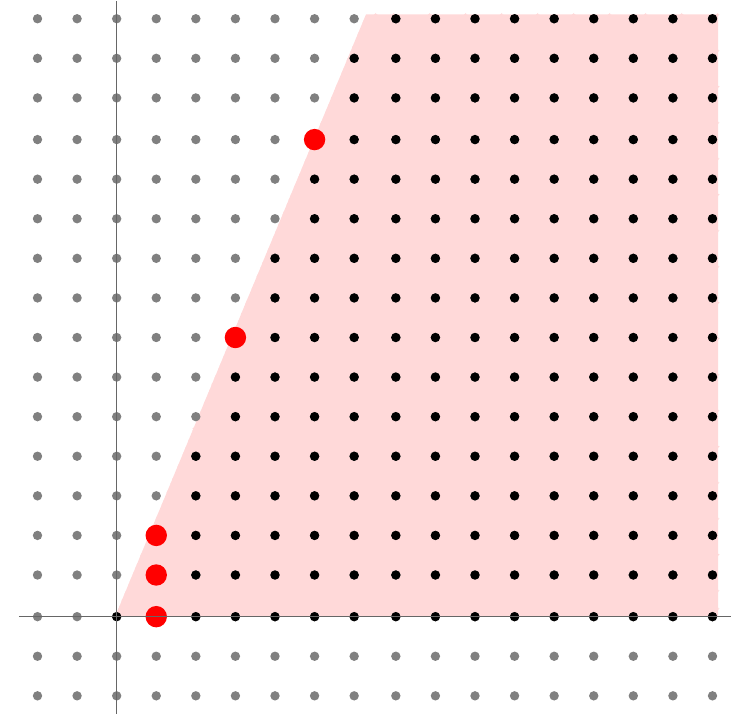}
\hspace{2cm}
\includegraphics[width=0.35\linewidth]{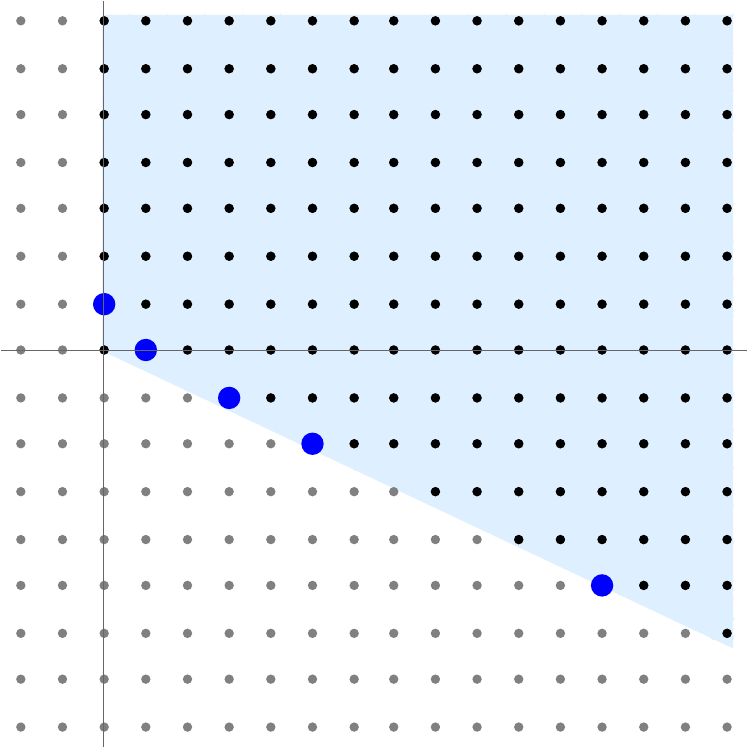}

\caption{Affine semigroup $\Gamma$ of Example~\ref{ex:lift} and its dual, with generators highlighted.} \label{fig:slackmatrixsemigroups}
\end{figure}

We can't have a lift with less than $5$ generators with dimension $2$. However, if we are allowed to increase the dimension, we can consider the affine semigroup $\lift{\Gamma} \subseteq \ZZ^3$ generated by the four elements $(1,0,0)$, $(6,0,1)$, $(0,2,1)$ and $(0,1,1)$.  One can check that the semigroup  $\lift{\Gamma}$ is in fact normal, that this is its minimal set of generators, and that the map $\iota:(x,y)\rightarrow (12,1,3)x-(5,0,1)y$ is a homomorphism from $\Gamma$ to $\lift{\Gamma}$. In fact, this induces an isomorphism between $\Gamma$ and the subsemigroup of $\lift{\Gamma}$ cut out by the plane spanned by $(12,1,3)$ and $(5,0,1)$, as can be seen in Figure \ref{fig:lift}.

\begin{figure}
\centering
\includegraphics[width=0.45\linewidth]{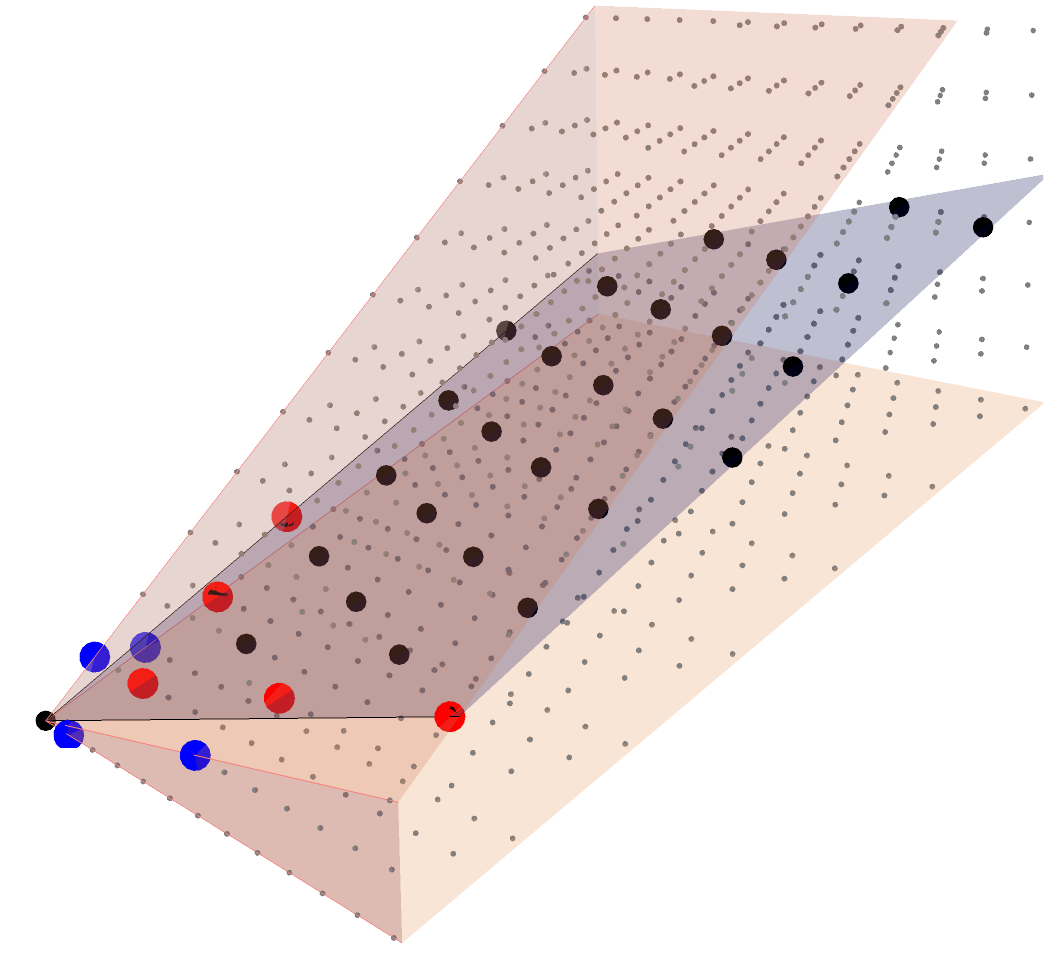}

\caption{Lift $\lift{\Gamma}$ of Example~\ref{ex:lift} with its generators highlighted in blue. The image of $\Gamma$ is highlighted in black with its generators highlighted in red. } \label{fig:lift}
\end{figure}

This map is in fact a lift. To see this one has to show that all the generators of $\Gamma^*$ can be obtained as pullbacks of elements in $\lift{\Gamma}^*$. For instance, the element $(0,1) \in \Gamma^*$ is simply the pullback of $(1,6,-6)$, since for $(x,y)\in\Gamma$
$$\langle (1,6,-6), \iota(x,y) \rangle = \langle (1,6,-6), (12x-5y,x,3x-y) \rangle = y = \langle (0,1), (x,y) \rangle. $$
To check $(1,6,-6) \in \lift{\Gamma}^*$ we just need to check its inner products with the four generators of $\lift{\Gamma}$ are nonnegative integers. Similarly $(1,0)$, $(3,-1)$, $(5,-2)$ and $(12,-5)$ are the pullbacks of $(0,1,0)$, $(0,0,1)$, $(0,-1,2)$ and $(1,0,0)$ respectively, and all of those lie in $\lift{\Gamma}^*$ so this is indeed a lift of $\Gamma$ with size $4$.
\end{ex}

In this case, we saw, as in Figure \ref{fig:lift}, that this lift is representing $\Gamma$ as a linear slice of $\lift{\Gamma}$, similar to the usual notion of lift in the convex cone case. In the next result we can see that this is not a particularity of this example.

\begin{thm}
If the affine semigroup $\lift{\Gamma}$ is a lift of the normal affine semigroup $\Gamma$ by some homomorphism $\iota$ then 
$\Gamma \cong \lift{\Gamma} \cap L$ where $L$ is a linear space such that $K_{\lift{\Gamma} \cap L}=K_{\lift{\Gamma}} \cap L$.
\end{thm}
\begin{proof}

Let $\lift{\Gamma}$ be a lift of $\Gamma$ by $\iota$.
We saw that $\iota$ is actually an injective map from $\Gamma$ to $\lift{\Gamma} \cap \spa(\iota(\Gamma))$. 
We claim that is is in fact a bijection. 

Let 
$\overline{x}=
\lift{\Gamma}\cap \spa(\iota(\Gamma))$. We wish to show that $\overline{x}$ is in the image of $\iota$. Let $a_1,\ldots a_n$ be the generators of $\Gamma$. Then $\overline{x}=\sum \alpha_i \iota(a_i)$ for some real coefficients $ \alpha_i$. Now consider $x=\sum \alpha_i a_i\in\spa(\Gamma)$. Let $b \in \Gamma^*$, then we have that $b = \iota^*(\overline{b})$ for some $\overline{b}\in\lift{\Gamma}^*$, so 
$$\langle b, x \rangle = \sum \alpha_i \langle b, a_i \rangle =  \sum \alpha_i \langle \overline{b}, \iota(a_i) \rangle = \langle\overline{b},\overline{x}\rangle \in \ZZ_+.$$
This means that $x \in (\Gamma^{*})^*=\overline{\Gamma}$. Since $\Gamma$ normal, this means $x \in \Gamma$, and so $\overline{x}=\iota(x) \in \iota(\Gamma)$ as stated. 
Extending $\iota$ in the usual way, $K_{\iota(\Gamma)} = \iota(K_\Gamma) $ and thus we are done by the same argument as before. 
\end{proof}

Note that, unlike in the convex case, the reverse of this statement is not true. Having $\Gamma \cong \lift{\Gamma} \cap L$ and $K_{\lift{\Gamma} \cap L}=K_{\lift{\Gamma}} \cap L$ is not enough to make $\lift{\Gamma}$ a lift of $\Gamma$. We need every homomorphism from $\lift{\Gamma} \cap L$ to $\ZZ_+$ to be obtainable as a restriction of an homomorphism from $\lift{\Gamma}$ to $\ZZ_+$, which is clearly a very restrictive condition.

The seminal result on lifts of polyhedral cones, is the theorem of Yannakakis \cite{Yann} that relates lifts of cones with nonnegative factorizations of their slack matrices. In the context of slack matrices of affine semigroups, we will be interested in \emph{nonnegative integer factorizations}. 

\begin{defn}\label{def:nnintrank}
Given a nonnegative integer matrix $A$, a \emph{nonnegative integer factorization} of $A$ is a factorization of it as $B^\top C$ where $B$ and $C$ are nonnegative integer matrices. The number of columns of $B$ (and $C$) is said to be the \emph{size} of the factorization. The smallest size among all nonnegative integer factorizations of $A$ is said to be the \emph{nonnegative integer rank} of $A$, denoted $\rankzn(A)$.
\end{defn}

The study of the nonnegative rank has long been a topic of interest in mathematics \cite{gregory1983semiring,COHEN93,berman1994nonnegative}. Nonnegative factorizations of matrices have a wide range of applications, and a good entry point to the large body of literature on the subject is the survey~\cite{gillis2020nonnegative}. 

The restriction of the factorizations to nonnegative integer matrices has also been studied in the combinatorial and linear algebra communities to a certain extent. It is a special case of a semiring rank or factorization rank, and as such it has been studied for example in \cite{gregory1983semiring,beasley1995rank}. However, most of the research seems to focus on either identifying classes of matrices for which certain ranks coincide (for example see \cite{beasley1995rank,BEASLEY198833,seok2006factor}) or restricting attention to $\{0,1\}$-matrices (for instance in \cite{GREGORY199173,doherty1999biclique,hefner1990minimum}), where it can be seen as the bipartite clique partition number of a bipartite graph. We will further explore the notion of  nonnegative integer rank in Section \ref{sec:nn_int_rank} but for now the definition is enough to state and prove a ``Yannakakis-type Theorem'' for affine semigroups.

 \begin{thm}[Yannakakis for affine semigroups] \label{thm:YanSemigroups}
 Let $\Gamma$ be a positive affine semigroup. The minimal size of a lift of $\Gamma$ equals the nonnegative integer rank of its slack matrix.
 \end{thm}
\begin{proof}
Let $\{a_1,...,a_n\}$ and $\{b_1,...,b_m\}$ be the generating sets for $\Gamma$ and $\Gamma^*$, respectively, and let $S_{\Gamma}$ be the $n \times m$ slack matrix of $\Gamma$.

We will start by proving that the rank of the matrix is less or equal than the minimal size of a lift.  Let $\iota : \Gamma \rightarrow \lift{\Gamma}$ be the lift of $\Gamma$ that has the smallest size, say $k$. Let $\{\bar{a}_1,...,\bar{a}_k\}$ and $\{ \bar{b}_1,...,\bar{b}_l\}$ be the generating sets for $\lift{\Gamma}$ and  $\lift{\Gamma}^*$, respectively, and let $S_{\lift{\Gamma}}$ be the $k \times l$ slack matrix of $\lift{\Gamma}$.

By the definition of lift, every $b_j$ is the pullback of some element in $\lift{\Gamma}^*$. This means that there exists nonnegative integers $c^j_p$ such that $b_j = \iota^*(\sum_p c^j_p \bar{b}_p)$. In particular
$$(S_{\Gamma})_{ij}=\langle a_i, b_j\rangle = \sum_{p=1}^l c^j_p \langle \iota(a_i),\bar{b}_p\rangle.$$
On the other hand $\iota(a_i)= \sum_{q} d_q^i \bar{a}_q$ for some nonnegative integers $d_q^i$ so we have
$$(S_{\Gamma})_{ij}= \sum_{p=1}^l  \sum_{q=1}^k c^j_p d_q^i \langle \bar{a}_q, \bar{b}_p \rangle = \sum_{p=1}^l  \sum_{q=1}^k c^j_p d_q^i (S_{\lift{\Gamma}})_{p,q}.$$
Defining $D$ and $C$ to be the $n \times k$ and the $m \times l$ nonnegative integer matrices whose entries are the $c^j_p$ and $d^i_q$ respectively, we have shown $S_{\Gamma} = D S_{\lift{\Gamma}} C$, which implies that the integer nonnegative rank is at most $k$.

We now need to show that if there is a a factorization of size $k$, we can find a lift of size at most $k$. Let $V$ and $W$ be $k\times n$ and $k \times m$ nonnegative integer matrices such that $S_{\Gamma}=V^\top W$. Consider the affine semigroup $\lift{\Gamma}=\Gamma^{\textup{row}}_W$. Note that $\Gamma^{\textup{row}}_{S_{\Gamma}} \subseteq \Gamma^{\textup{row}}_W$ so, if we denote by $S_i$ the $i$-th row of $S_{\Gamma}$ we can define an homomorphism
$\iota: \Gamma \rightarrow \lift{\Gamma}$ by setting $\iota(a_i) = S_i$. Note that we have seen in Lemma \ref{lem:isomorphismslack} that this is in fact an injective homomorphism. Since $\lift{\Gamma}$ is generated by $k$ elements, to conclude the proof we just need to show that it is indeed a lift, that is, $\iota^*(\lift{\Gamma}^*) = \Gamma^*$.

First we will show that $\Gamma^* \subseteq \iota^*(\lift{\Gamma}^*)$. We claim that $\iota^*(e_j)=b_j$. To see this, first note that $e_j\in\lift{\Gamma}^*$ because $\lift{\Gamma} \subseteq \ZZ_+^m$ and the operator $e_j$ simply takes the $j$-th entry of a vector. Moreover, for each generator $a_i$ of $\Gamma$ we have
$$\langle  a_i, \iota^*(e_j)\rangle = \langle  \iota(a_i), e_j\rangle = \langle  S_i,e_j\rangle =S_{ij} = \langle a_i, b_j \rangle.$$
Thus $\iota^*(e_j)=b_j$ for each $j$, so that  all the generators of $\Gamma^*$ are in $\iota^*(\lift{\Gamma}^*)$. 

The other inclusion is an immediate consequence of $\iota(\Gamma) \subseteq \lift{\Gamma}$, since if $x \in \lift{\Gamma}^* $ and $y \in \Gamma$, then $\langle y, \iota^*(x)\rangle = \langle \iota(y), x\rangle \in \ZZ_+$.
\end{proof}

\begin{ex}
In Example \ref{ex:lift} we saw a lift of the normal affine semigroup $\Gamma$ generated by $(1,0)$, $(5,12)$, $(3,7)$, $(1,2)$, and $(1,1)$. We also saw the generators of its dual, $(0,1)$, $(12,-5)$, $(1,0)$, $(3,-1)$, and $(5,-2)$, so combining the information we can see that its slack matrix is
$$S_{\Gamma}=\begin{bmatrix}
0& 12& 1 &3& 5 \\ 12& 0& 5& 3& 1\\ 7& 1& 3& 2& 1\\ 2& 2& 1& 1& 1\\ 1& 7& 1& 2& 3     
\end{bmatrix}.$$
By Theorem \ref{thm:YanSemigroups}, since $\Gamma$ has a lift of size $4$, there must be an nonnegative integer factorization of this matrix with inner dimension $4$, and in fact
$$S_{\Gamma}=\begin{bmatrix}
0 & 2 & 0 & 1\\ 0& 0 & 2 & 1 \\ 1 & 0 & 1 & 1\\ 2 & 0 & 0 & 1 \\ 1 & 1 & 0 & 1
\end{bmatrix}
\begin{bmatrix}
1 & 1 & 0 & 0 & 0 \\ 0 & 6 & 0 & 1 & 2 \\ 6 & 0 & 2 & 1 & 0\\ 0 & 0 & 1 & 1 & 1
\end{bmatrix}.$$
\end{ex}

\subsection{Connection to integer linear programming}

Yannakakis' original theorem for convex polyhedral cones is celebrated for its important applications to complexity questions in linear programming. A natural question to ask is if our affine semigroup version has analogous implications to integer linear programming.

To that end, we consider the integer linear program
\begin{equation}
    \begin{array}{rcl}
    s = & \min_x & c^\top x \\[3pt]
      & \text{s. t. } & A x \geq 0  \\[3pt]
      & &d^\top x = e \\[3pt]
      & & x \in \ZZ^n\\ [3pt]
\end{array} \label{EQ:int_prog}
\end{equation}
for $c \in \ZZ^n, A \in \ZZ^{m \times n}, d \in \ZZ^n$ and $e \in \ZZ$.
Denote the rows of $A$ by $a_i, i=1,\ldots m$. If we let $\Gamma=\Gamma_{A}^{\textup{row}}$ and assume $\Lambda_A^{\textup{row}}=\ZZ^n$, we can write $c$ and $d$ as integer combinations of the rows~$a_i$. Then \eqref{EQ:int_prog} becomes
\begin{equation}
    \begin{array}{rcl}
        s = & \min_x & \sum_{i=1}^m c_i \langle a_i, x\rangle \\[3pt]
      & \text{s. t. } & \sum_{i=1}^m d_i \langle a_i, x\rangle = e \\[3pt]
      & & x \in \Gamma^* \\[3pt]
\end{array} \label{EQ:int_prog_2}
\end{equation}
for some $c_i,d_i\in\ZZ$.
The size of the generating set of $\Gamma$ corresponds to the number of constraints of the original program, while its dimension corresponds to the number of variables. We will show that lifts of $\Gamma$ can be used to rewrite the ILP \eqref{EQ:int_prog_2}.

Given a lift $\iota:\Gamma\to\lift{\Gamma}$ consider the ILP
\begin{equation}
    \begin{array}{rcl}
       s = & \min_x & \sum_{i=1}^m c_i \langle \iota(a_i), y\rangle \\[3pt]
      & \text{s. t. } & \sum_{i=1}^m d_i \langle \iota(a_i), y\rangle = e \\[3pt]
      & & y \in \lift{\Gamma}^* \\[3pt]
\end{array} \label{EQ:int_prog_lift}
\end{equation}
This has the same form as  \eqref{EQ:int_prog_2}, and we can see that it is equivalent as follows.

\begin{prop}
If  $\iota:\Gamma\to\lift{\Gamma}$ is a affine semigroup lift, then 
\eqref{EQ:int_prog_2} and \eqref{EQ:int_prog_lift} have the same minimum value.
\end{prop}
\begin{proof}
Let $\tilde{x}$ be a solution to \eqref{EQ:int_prog_2}. By the definition of lift there exists $\tilde{y} \in \lift{\Gamma}^*$ such that $\iota^*(\tilde{y})=\tilde{x}$. It is clear that $\tilde{y}$ verifies the constraint in \eqref{EQ:int_prog_lift} and has the same objective value as \eqref{EQ:int_prog_2}. On the other hand, if $\hat{y}$ is a solution of \eqref{EQ:int_prog_lift}, then setting $\hat{x}=\iota^*(\hat{y})$, by definition of lift we have that $\hat{x} \in \Gamma^*$. It is again clear that $\hat{x}$ verifies the constraint in \eqref{EQ:int_prog_2} and has the same objective value as \eqref{EQ:int_prog_lift}.
\end{proof}

In this sense, we see that lifts of affine semigroups give rise to lifts of ILPs just as in the LP case. While that is somewhat satisfying, its usefulness for complexity questions is limited. In the LP case the hardness of a problem depends essentially only on the number of constraints, with the number of variables being largely irrelevant. However, in the ILP case the complexity grows as $\log(n)^n$ in the number of variables and is polynomial in the number of constraints and encoding length, so modestly increasing the dimension to reduce the number of constraints is not generally appealing (see \cite{Reis2023TheSF} for a recent breakthrough on the complexity of ILP). While a more compact expression could perhaps be useful in some niche circumstances, there do not seem to exist direct complexity implications from the existence of lifts in this particular encoding. 

\subsection{Connection to toric varieties}

Affine semigroups are intimately connected with toric varieties. Recall that given a finite set $\mathcal{A}=\{a_1,...,a_s\} \subseteq \ZZ^n$ one can define the affine toric variety $Y_{\mathcal{A}}$ as the Zariski closure of the image of the map $\phi_{\mathcal{A}}:(\CC^*)^n \rightarrow (\CC^*)^s$ given by $\phi_{\mathcal{A}}(x)=(x^{a_1},x^{a_2},...,x^{a_s})$.

The dimension of the variety $Y_{\mathcal{A}}$ is the dimension of the span of $\mathcal{A}$, while the embedding dimension is the number $s$ of elements of $\mathcal{A}$. The connection between affine toric varieties and affine semigroups is codified in the following proposition, a proof of which can be found in \cite[Theorem 1.1.17]{CLS11}.

\begin{prop}
If $\mathcal{A}$ is the set of minimal generators of an affine semigroup $\Gamma$, then $Y_{\mathcal{A}} = \textup{Spec}(\CC[\Gamma])$. Moreover, every toric variety can be obtained in this way.
\end{prop} 

Since the dual of an affine semigroup $\Gamma$ can be identified with linear maps that take nonnegative integer values on $\Gamma$, it follows that in the setting of toric varieties we can identify elements of $\Gamma^*$ with gradings on the coordinate ring $\CC[\Gamma]$ of the variety $Y_{\mathcal{A}}$.

\begin{lem}
    Let $\Gamma$ be an affine semigroup. The elements of $\Gamma^*$ correspond to $\NN$-gradings of the ring $\CC[\Gamma]$.
    \label{lem:dual_gives_gradings}
\end{lem}

\begin{proof}
   Let $\mathcal{A} = \{a_1,\ldots, a_s\}$ generate $\Gamma$. Then recall that $\CC[\Gamma] = \CC[\chi^{a_1},\ldots, \chi^{a_s}]$. If $y\in\Gamma^*$, then for every $m\in\Gamma$, we have $\langle m,y\rangle\in\ZZ_+$. Thus $y$ gives an $\NN$-grading of $\CC[\Gamma]$ where 
\[
\CC[\Gamma]_n = \left\{\sum_{m\in\Gamma} c_m\chi^m \;|\; c_m\in\CC, \langle m,y\rangle = n \right\}.
\]

Conversely, given an $\NN$-grading of $\CC[\Gamma]$, for each $i=1,
\ldots s$, $\chi^{a_i}$ lives in some graded piece, say of degree $n_i$. Since each $m\in\Gamma$ can be written $m = \lambda_1a_1+\cdots+\lambda_sa_s$ for $\lambda_i\in\ZZ_+$, by the properties of a grading, we have that $\chi^m$ has degree $\sum_i\lambda_in_i\in\ZZ_+$. Thus the grading gives us a homomorphism from $\Gamma$ to $\ZZ_+$, and hence corresponds to an element of $\Gamma^*$. 
\end{proof}

Now a semigroup homomorphism $\psi:\Gamma_1 \rightarrow \Gamma_2$, induces a homomorphism $\overline{\psi}:\CC[\Gamma_1]\to\CC[\Gamma_2]$ which, in turn, gives a morphism $\hat\psi: \textup{Spec}(\CC[\Gamma_2])\to \textup{Spec}(\CC[\Gamma_1])$. A morphism of toric varieties induced in this way is precisely what is known as a \emph{toric morphism}.

This means that a lift $\iota : \Gamma \rightarrow \lift{\Gamma}$ corresponds to a toric morphism $\hat{\iota}:Y_{{\mathcal{B}}}\to Y_{{\mathcal{A}}}$, where $\mathcal{A}$ and ${\mathcal{B}}$ are the generating sets of $\Gamma$ and $\lift{\Gamma}$ respectively. Then since $\iota^*(\lift{\Gamma}^*) = \Gamma^*$, we also get a relation on the gradings of the corresponding coordinate rings.

\begin{defn} 
    Let $V$ be a toric variety. We call a toric variety $W$ a {\em grading lift} of $V$, if there is a toric morphism $\psi:W \to V$ such that for each $\NN$-grading of $\CC[V]$ there exists an $\NN$-grading of $\CC[W]$ such that the induced map $\overline{\psi}$ satisfies $\psi(\CC[V]_n)\subset\CC[W]_n$ for all $n\geq 0$.
\end{defn}

\begin{thm}
    Let $\Gamma$ be an affine semigroup with generating set $\mathcal{A}$, and let $Y_{\mathcal{A}} = \textup{Spec}(\CC[\Gamma])$ be the corresponding toric variety. There exists a grading lift $Y_\mathcal{B}$ of $Y_{\mathcal{A}}$ if and only if there exists a lift $\lift{\Gamma}$, generated by $\mathcal{B}$, of $\Gamma$. 

    The minimal embedding dimension in which such a lift of $Y_\mathcal{A}$ exists is given by the nonnegative integer rank of a slack matrix of $\Gamma$.
\end{thm}

\begin{proof}
First, suppose we have a lift $\iota:\Gamma\to\lift{\Gamma}$. Then every $\NN$-grading of $\CC[\lift{\Gamma}]$ maps to an $\NN$-grading of $\CC[\Gamma]$ via $\iota^*$ and Lemma~\ref{lem:dual_gives_gradings}. Then from the lift property, given a grading $\gamma\in\Gamma^*$ of $\CC[\Gamma]$, we have $\gamma= \iota^*(\delta)$ for some $\delta\in\lift{\Gamma}^*$, and  
\[
\CC[\Gamma]_n = \left\{\sum_{a\in\Gamma} c_a\chi^a \;|\; c_a\in\CC, \langle a,\iota^*(\delta)\rangle = n \right\}.
\]
Denote by $\overline{\iota}$ the map from $\CC[\Gamma]$ to $\CC[\lift{\Gamma}]$ induced by $\iota$. 
Since $\langle a,\iota^*(\delta)\rangle = \langle \iota(a),\delta\rangle$, we get that $\delta$ defines an $\NN$-grading of $\CC[\lift{\Gamma}]$ such that degree is preserved by $\overline{\iota}$
\[
\overline{\iota}\big(\CC[\Gamma]_n\big) = \left\{\sum_{a\in\Gamma} c_a\chi^{\iota(a)} \;|\; c_a\in\CC, \langle \iota(a),\delta\rangle = n \right\} \subseteq \CC[\lift{\Gamma}]_n.
\]

Conversely, suppose we have a grading lift $Y_\mathcal{B}$ of $Y_\mathcal{A}$ with toric morphism $\hat\psi: Y_{\mathcal{B}} \to Y_\mathcal{A}$, where $\mathcal{A,B}$ generate affine semigroups $\Gamma,\lift{\Gamma}$, respectively. That is, for each $\NN$-grading of $\CC[\Gamma]$, there is an $\NN$-grading of $\CC[\lift{\Gamma}]$ such that the corresponding map of coordinate rings $\overline{\psi}$ satisfies 
\begin{equation}
 \overline{\psi}\left( \CC[\Gamma]_n \right) \subset \CC[\lift{\Gamma}]_n \text{ for all } n\geq 0.
 \label{eq:deg_preserving_morphism}
\end{equation}

Since $\hat{\psi}$ is a toric morphism, we know that $\overline{\psi}$ is induced by a semigroup homomorphism $\psi:\Gamma\to\lift{\Gamma}$.

Then given $\gamma\in\Gamma^*$, by \eqref{eq:deg_preserving_morphism}, there exists $\delta\in\lift{\Gamma}^*$ such that 
\begin{align*}
\overline{\psi}\big(\CC[\Gamma]_n\big) 
    & = \left\{\sum_{a\in{\Gamma}} c_a\chi^{\psi(a)} \;|\; c_a\in\CC, \langle a,\gamma\rangle = n \right\} \\
    & \subseteq \left\{\sum_{b\in{\lift{\Gamma}}} c_b\chi^{b} \;|\; c_b\in\CC, \langle b,\delta\rangle = n \right\}.
\end{align*}
That is, for each $n\geq 0$, $\langle a, \gamma\rangle = n$ implies $\langle \psi(a), \delta\rangle = n$. But now 
\[
\langle a, \gamma\rangle = \langle\psi(a),\delta\rangle = 
\langle a,\psi^*(\delta)\rangle
\]
for all $a\in\Gamma$, so we must have $\gamma = \psi^*(\delta)$, so that $\psi:\Gamma\to \lift{\Gamma}$ is a lift.

The final statement follows from Theorem~\ref{thm:YanSemigroups} and the fact that embedding dimension of a toric variety $Y_\mathcal{B}$ is given by the number of elements in $\mathcal{B}$.

\end{proof}

\section{Nonnegative integer rank} \label{sec:nn_int_rank}

In Definition   \ref{def:nnintrank} we introduced the concept of nonnegative integer rank of a nonnegative integer matrix. In this section, we survey some of the known properties for this notion of rank, as well as establish some new results.

The nonnegative rank is a special case of what is known as a semiring or a factorization rank (see \cite{gregory1983semiring}). Given $A \in \ZZ_+^{n \times m}$, for any semiring $S$ that contains $\ZZ_+$ one can define a $\rank_S(A)$ as the smallest $k$ for  which there exist $B \in S^{n \times k}$ and $C \in S^{k \times m}$ such that $A=BC$. 

For $S=\ZZ_+$ we recover $\rankzn(A)$, while for $S=\RR_+$ we get the usual nonnegative rank $\rankp(A)$, and for $S=\RR$ we simply obtain the rank. From this simple observation it is clear that we get the following result, which can be seen in a much more general form in \cite{beasley1995rank}.

\begin{lem}
For any matrix $A \in \ZZ_+^{n \times m}$, we have
$$\rank(A) \leq \rankp(A) \leq \rankzn(A).$$
\end{lem}

Besides this obvious bound, there is a wealth of literature in lower bounding the nonnegative rank (see for instance \cite{fiorini2013combinatorial,fawzi2015lower}) that trivially give lower bounds on $\rankzn(A)$ as well. But it seems that there are no general lower bound inequalities that apply specifically to $\rankzn(A)$.

It is well known that for $\rank(A)\leq 2$ we must have $\rankp(A)=\rank(A)$, while for any other value of $\rank(A)$ we can have a strict inequality (see \cite[Corollary 4.5.1]{beasley1995rank} for a refined version of this result). However, even for rank $2$ we can find a gap for the nonnegative integer rank.

\begin{ex}
Consider following rank $2$ nonnegative integer matrix
$$A=\begin{pmatrix}
2 & 1 & 0 \\
1 & 1 & 1 \\
0 & 1 & 2
\end{pmatrix} =
\begin{pmatrix}
1 & 0 \\
\frac{1}{2} & \frac{1}{2}   \\
0 & 1
\end{pmatrix}
\begin{pmatrix}
2 & 1 & 0 \\
0 & 1 & 2 
\end{pmatrix}.
$$
The given factorization shows that $\rankp(A)=\rank(A)=2$, but we claim that $\rankzn(A)=3$. To see this, note that we can think of $\rankzn(A)$ as the smallest possible number of matrices of the form $v_i w_i^\top$, where $w_i,v_i \in \ZZ_+^3$, that sum up to $A$. Each individual matrix of this form  necessarily has rectangular support, and the only way of covering the support of $A$ with two rectangles would require two rectangles overlapping at the central entry. Since each $v_i w_i^\top$ has nonnegative integer entries, that would imply the central entry must be at least $2$. 
% AN ALTERNATE EXAMPLE CAN BE FOUND IN \CITE[EXAMPLE 4.2]{BEASLEY1995RANK} AND IS THE MATRIX 
% $$B=\BEGIN{PMATRIX}
% 2 & 0 & 3 \\
% 1 & 1 & 4 \\
% 1 & 3 & 9
% \END{PMATRIX}.$$ 
\end{ex}

One interesting question is how large can the gap between $\rankp(A)$ and $\rankzn(A)$ be. For the $\{0,1\}$-case the paper \cite{MR1160062} establishes that the ratio ${\rank(A)}/{\rankzn(A)}$ can be arbitrarily close to zero, but there seem to be no implications in terms of the ratio ${\rankp(A)}/{\rankzn(A)}$. In what follows we establish that this ratio can also be arbitrarily close to zero and, in fact, this is true even when restricting to $\rankp(A)=2$. 

To find an example of such behaviour, we use the theory we have developed regarding slack matrices of affine semigroups. The intuition for searching for such examples among slack matrices is that being a slack matrix imposes significant constraints on nonnegative factorizations, so it should be easier to find examples with dramatic gaps between nonnegative rank and integer nonnegative rank. It turns out that this does allow us to find examples with the desired gap. 

\subsection{Fibonacci semigroups and nonnegative integer rank}

For every positive integer $n$ consider the normal affine semigroup
$$\Gamma^{\textup{Fib}}_n= \ZZ^2 \cap \{(x,y) \in \RR^2 \, : \,  0 \leq F_{2n} y \leq F_{2n+2} x\}$$
where $F_n$ is the $n$-th element of the  Fibonacci sequence $(1,1,2, \ldots )$. In other words, $\Gamma^{\textup{Fib}}_n$  is the set of integer points in a cone in the positive orthant bounded by the $x$-axis and the line with slope $F_{2n+2}/F_{2n}$. We will see some interesting properties of this family of affine semigroups, starting with a few simple properties that are implied by basic properties of the Fibonacci sequence.

\begin{lem} \label{lem:fibsemigroupprop}
Let $\left\{\Gamma^{\textup{Fib}}_n\right\}_{n\geq 1}$ be the sequence of semigroups defined above. Then
\begin{enumerate}
\item We have $\Gamma^{\textup{Fib}}_k \supset \Gamma^{\textup{Fib}}_{k+1}$ for all $k$;
\item $\bigcap_{n =1}^{\infty} \Gamma^{\textup{Fib}}_n = \ZZ^2 \cap \{(x,y) \in \RR^2 \, : \, 0 \leq  y \leq \phi^2 x\}$
with $\phi$ the golden ratio;
\item The dual of $\Gamma^{\textup{Fib}}_n$ is $(\Gamma^{\textup{Fib}}_n)^* = \ZZ^2 \cap \{(x,y) \in \RR^2 \, : \,  0 \leq x, F_{2n+2} y \geq - F_{2n} x\}$.
\end{enumerate}
\end{lem}
\begin{proof}
The first fact is basically equivalent to say that the sequence of slopes $a_n=F_{2n+2}/F_{2n}$ is decreasing. But $$ \frac{a_{n}}{a_{n-1}}=\frac{F_{2n+2}F_{2n-2}}{F_{2n}^2}$$ and by Catalan's identity, $F_{2n+2}F_{2n-2}=F_{2n}^2-1$, so $\frac{a_{n}}{a_{n-1}} < 1$ and we have the desired monotonicity.

For the second fact, it is enough to note that $\lim_n a_n = \phi^2$, since 
$$\lim_n a_n = \left(\lim_n \frac{F_{2n+2}}{F_{2n+1}}\right)\left(\lim_n \frac{F_{2n+1}}{F_{2n}}\right) $$
and the limit of the quotient of consecutive Fibonacci numbers is the golden ratio.

Finally, the last fact is simply obtained applying the definition of the dual. The underlying lattice is $\ZZ^2$, which is self dual, and the dual of the cone can be easily seen to be the cone in the statement.
\end{proof}

We are interested in the slack matrix of this semigroup. We can also compute the general form of these matrices using properties of the Fibonacci numbers.

\begin{lem}
The minimal generating sets of $\Gamma^{\textup{Fib}}_n$ and  $(\Gamma^{\textup{Fib}}_n)^*$ are, respectively,
$$\mathcal{A}=\{(F_{2n},F_{2n+2}),(1,0),(1,1)\} \cup \{(F_{2k-1},F_{2k+1}), k=1,...,n\}$$
and
$$\mathcal{B}=\{(0,1)\} \cup \{(F_{2k+2},-F_{2k}), k=0,...,n\}.$$
With this ordering the slack matrix is the $(n+3) \times (n+2)$ matrix
$$S_{n}^\ZZ=\begin{bmatrix}
F_{2n+2}         & F_{2n}     & F_{2n-2}   & F_{2n-4} & \cdots    & F_{2}    & 0       \\
0                & F_{2}      & F_{4}      & F_6      & \cdots    & F_{2n}   & F_{2n+2} \\
F_{1}            & F_{1}      & F_{3}      & F_5      & \cdots    & F_{2n-1} & F_{2n+1} \\
F_{3}            & F_{1}      & F_{1}      & F_3      & \cdots    & F_{2n-3} & F_{2n-1}  \\
F_{5}            & F_{3}      & F_{1}      & F_1      & \cdots    & F_{2n-5} & F_{2n-3}  \\
\vdots           & \vdots     & \ddots     & \ddots   &           & \ddots   & \vdots \\
F_{2n+1}         & F_{2n-1}   & F_{2n-3}   & F_{2n-5} & \cdots    & F_1       & F_1

 \end{bmatrix}$$
\end{lem}

\begin{proof}

That $\mathcal{A}$ and  $\mathcal{B}$ are Hilbert basis, can be seen from two facts.

The first is that the convergents of the continued fraction expansion of $\frac{F_{2n+2}}{F_{2n}}$ are the quotients $F_{k+2}/F_k$ for $k \leq {2n+2}$  as it is simply a truncation of the continued fraction of $\phi^2$, which has these as convergents. 

The second is that the set $\{(x,y) \in \ZZ^2_+ \, : \, x \geq \alpha y\}$ is generated by the vectors $(p,q)$ where $p/q$ is an upper convergent to $\alpha$. In this form, this result can be found for instance in \cite[Theorem 1]{DRESS1982292}.

By setting $\alpha=\phi^{-2}$, and noting that inverting the number only adds a zero to the list of convergents and inverts the others we get the result for $\mathcal{A}$. A simple transformation of the second set by changing the sign of the second coordinate allow us to get the result for~$\mathcal{B}$. 

Now that we have the generators we simply need to calculate the slack matrix. The simplified form presented here comes essentially from d'Ocagne's identity: $F_mF_{n+1}-F_{m+1}F_n=(-1)^nF_{m-n}$. In fact by expanding $F_{2k+1}$ and $F_{2l+2}$ with the Fibonacci recursion formula we get
$$(F_{2k-1},F_{2k+1})^\top (F_{2l+2},-F_{2l}) =  F_{2k-1}F_{2l+1}-F_{2k}F_{2l} = F_{2(k-l)-1}.$$
Here we follow the usual convention $F_{-m}=(-1)^{m+1}F_m$. All the other entries either follow the same idea or are trivial.
\end{proof}

Note that we have $\rank(S_{n}^\ZZ)=2$ which means that we must also have  $\rankp(S_{n}^\ZZ)=2$ for every $n$. A brute force computational search shows that for $n =1,...,4$ we have $\rankzn(S_{n}^\ZZ)=n+2$, the maximum possible.
However, even in these small examples certifying the computed integer nonnegative rank in a human-readable form is hard; each case seems to require different and sometimes long ad hoc arguments.

Although the question of certifying the exact rank of these matrices remains open, we are at least able to show that the rank grows to infinity as $n$ grows. To do this we will focus on the submatrix of $S_n^{\ZZ}$ obtained by dropping the first column and the first two rows. This is a $(n+1) \times (n+1)$ matrix, that we will denote by $M_n$. It is a Toeplitz matrix, with $(i,j)$-entry  equal to $F_{2(i-j)+1}$. So we have
\begin{equation}
    M_n=\begin{bmatrix}
F_{1}         & F_{3}      & F_{5}      & \cdots      & \cdots  & F_{2n+1}       \\
F_{1}         & F_{1}      & F_{3}      & \ddots      &         & \vdots \\
F_{3}         & F_{1}      & \ddots     &  \ddots     &  \ddots &  \vdots  \\
\vdots        &\ddots      & \ddots     & \ddots      &   F_{3} & F_{5}  \\
\vdots        &            & \ddots     & F_{1}        & F_{1}  & F_{3}  \\
F_{2n-1}      & \cdots     & \cdots     & F_{3}         & F_{1}       & F_1
 \end{bmatrix}.\label{eq:fibmatrix}
\end{equation}
This is a submatrix of the rank 2 integer Toeplitz matrix with entries $F_{i-j}$, which has been studied for instance in \cite{dixon2014application}. Here we prove a simple fact about the growth of its nonnegative integer rank using its recursive structure. To do that we will use the following lemma that gives sufficient conditions to guarantee that a matrix rank is bigger than that of certain submatrices.

\begin{lem} \label{lem:normbound}
Suppose $A_1 \in \ZZ_+^{n_1 \times m_1}$ and $A_2 \in \ZZ_+^{n_2 \times m_2}$ have the same nonnegative integer rank $r$. Let $A \in \ZZ_+^{(n_1+n_2) \times (m_1 + m_2)}$ have the form
$$A=\begin{bmatrix} A_1 & \star \\ \star & A_2 \end{bmatrix}.$$
If $(\|A_1\|_1+2)(\|A_2\|_1+2) \leq \|A\|_1+4$, where $\| \cdot \|_1$ represents the sum of the entries of the matrix, then $\rankzn(A) \geq r+1$. 
\end{lem}

\begin{proof}
We know $\rankzn(A) \geq r$, so suppose for a  contradiction it is $r$.
Then there are vectors $a_i \in \ZZ_+^{n_1}$, $b_i \in \ZZ_+^{n_2}$, $c_i \in \ZZ_+^{m_1}$ and $d_i \in \ZZ_+^{m_2}$ for $i=1,...,r$ such that 
$$A=\sum_{i=1}^r \begin{pmatrix}
    a_i \\ b_i
\end{pmatrix}(c_i^\top,d_i^\top).$$ this implies $\sum_{i=1}^r a_ic_i^\top = A_1$ and $\sum_{i=1}^r b_i d_i^\top = A_2$, and since $r$ is the nonnegative integer rank of $A_1$ and $A_2$, none of these summands can be zero.

In particular, we have 
\begin{equation}\label{eq:sumoftermsA}
\sum_{i=1}^r \|a_i\|_1\|c_i\|_1= \|A_1\|_1, \textrm{ and } \sum_{i=1}^r \|b_i\|_1\|d_i\|_1 =  \|A_2\|_1 
\end{equation}
and none of the norms can be zero. Similarly, from the factorization of $A$ we get
\begin{equation}\label{eq:sumoftermsB}
 \sum_{i=1}^r (\|a_i\|_1+\|b_i\|_1)(\|c_i\|_1+\|d_i\|_1) =  \|A\|_1.
\end{equation}
We will compute the maximum value the sum in \eqref{eq:sumoftermsB} can have, given \eqref{eq:sumoftermsA} as constrains. To simplify notation, we will replace the norms by integer unknowns that are greater or equal to $1$ and we get the optimization problem
\begin{equation} \label{eq:maxtermsA}
\textup{max}_{\ZZ_{++}} \sum_{i=1}^r  ({\alpha_i}+{\beta_i})({\gamma}_i+{\delta}_i) \textup{ s.t. }  \sum_{i=1}^r \alpha_i \gamma_i = \|A_1\|_1   \textrm{ and } \sum_{i=1}^r \beta_i \delta_i = \|A_2\|_1.
\end{equation}
Now we note  that given any feasible solution $(\alpha,\beta,\gamma,\delta)$ of \eqref{eq:maxtermsA},  making the substitutions $\hat{\alpha}_i=\alpha_i \gamma_i$, $\hat{\delta}_i=\beta_i \delta_i$, $\hat{\gamma}_i=\hat{\beta}_i=1$, keeps it feasible, as we preserved the products $\hat{\alpha_i}\hat{\gamma}_i$ and $\hat{\beta_i}\hat{\delta}_i$.
Moreover, one can see that
$$(\hat{\alpha}_i+\hat{\beta}_i)(\hat{\gamma}_i+\hat{\delta}_i)=(\alpha_i \gamma_i +1)(1+\beta_i \delta_i) \geq ({\alpha_i}+{\beta_i})({\gamma}_i+{\delta}_i)$$
by expanding the products and noting that the sum of two positive integers is always at most their product plus one. So the maximum value for  \eqref{eq:maxtermsA} is attained when all $\beta_i$ and $\gamma_i$ are $1$. By setting $v_i=\hat{\alpha}_i+1$ and $w_i =\hat{\delta}_i+1$, 
we transform \eqref{eq:maxtermsA} into the equivalent problem
\begin{equation} \label{eq:maxtermsB}
\textup{max}_{\ZZ_{++}} v^\top w \textup{ s.t. }  \|v\|_1=\|A_1\|_1+r,   \|w\|_1=\|A_2\|_1+r \textup{ and } v_i \geq 2, w_i \geq 2, i=1,...,r. 
\end{equation}
Maximizing the inner product is always less or equal than maximizing the product of the $2$-norms. Moreover, the maximum $2$-norm of a vector with given $1$-norm and entries all greater or equal to $2$ can be found explicitly, as it is achieved when all entries except possibly one are equal to $2$. From this, we see that the maximum is upper bounded by
$$\sqrt{(4(r-1)+(\|A_1\|_1-r+2)^2)(4(r-1)+(\|A_2\|_1-r+2)^2)}.$$
One can check that this is strictly smaller than $M := (\|A_1\|_1+2)(\|A_2\|_1+2)-4$. Thus if $\|A\|_1\geq M$, the equations \eqref{eq:sumoftermsB} and \eqref{eq:maxtermsB} are incompatible, giving us the result.
\end{proof}

We are now ready to state our growth result.

\begin{prop} For $M_n$ as defined in \eqref{eq:fibmatrix}, we have $\rankzn(M_{2n+1}) \geq \rankzn(M_n)+1$ for each $n \geq 1$.
\end{prop}
\begin{proof}

Now note that $M_{2n+1}$ has block structure
$$M_{2n+1}=\begin{bmatrix}
M_n & A \\
B & M_n
\end{bmatrix}.$$

Using standard formulas for the sum of the first consecutive odd or even indexed terms of the Fibonacci sequence one can easily show that for any $k$ the sum of the entries of $M_k$ is $F_{2k+1}+F_{2k+3}-2$. So using Lemma \ref{lem:normbound} it is enough to show that
$$(F_{2n+1}+F_{2n+3})^2 \leq F_{4n+3}+F_{4n+5}+2.$$
But the left hand side is  $F_{2n+1}^2 + 2F_{2n+1}F_{2n+3} +F_{2n+3}^2$
and, by Cassini's formula, $F_{2n+1}F_{2n+3}= F_{2n+2}^2+1$. Applying the standard formula for the sum of the squares of two consecutive Fibonacci numbers we get that the inequality holds with equality, proving the result.
\end{proof}

This can immediately gives us some simple corollaries. The first one concerns gaps between ranks.

\begin{cor}
For nonnegative integer matrices $M$, $\frac{\rankp(M)}{\rankzn(M)}$ can be arbitrarily close to zero, even if we restrict to $\rankp(M)=2$.
\end{cor}

The second concerns lifts of affine semigroups.

\begin{cor}
There exist $2$-dimensional affine semigroups with minimal lifts of arbitrarily high size.
\label{cor:semilift}
\end{cor}

Note that the second Corollary implies the first. It is very plausible that there is a simple alternative proof for this second Corollary using only the geometry of affine semigroups, and not the slack matrices themselves. 

\section{Conclusion}
\label{sec:concl}

The goal of this paper was to expand the notion of lifts of convex sets to the discrete setting. Inspired by Yannakakis' original theorem on lifts of polyhedral sets, we started by showing that there is a natural generalization of the slack matrix of a cone to the setting of affine semigroups. We demonstrated that these slack matrices are also interesting objects in their own right. Indeed, they share many properties with the original notion of a slack matrix providing, in particular, a canonical embedding of (the equivalence class of) its affine semigroup with special properties. We then showed that there is a sensible notion of lift for an affine semigroup that satisfies a Yannakakis-type theorem.

This initial result suggests a number of different avenues of research that we only very briefly began to explore in this paper. The wealth of connections between affine semigroups and toric geometry and integer programming suggests that there might be interesting translations of this result to those settings. While we provide an initial sketch of how those translations might look, there seems to be room for a much deeper exploration of those connections.

Studying the structure of integer slack matrices also allowed us to produce new results regarding nonnegative integer rank of matrices, namely, that even in the case of rank 2 matrices, we can get arbitrarily large gap between nonnegative rank and integer nonnegative rank. In this direction too some questions remain. Computational evidence showed that the slack matrix of the Fibonacci semigroup $\Gamma_n^{\textup{Fib}}$ has nonnegative integer rank $n+2$, which is a much stronger result than what we showed in Proposition~\ref{lem:fibsemigroupprop}. It remains to show whether this stronger result holds for $n\geq 5$. Additionally, one might look for a geometric argument using the structure of the semigroups to prove the result of Corollary~\ref{cor:semilift}.

Other questions remain for further exploration. For instance, our motivation from integer programming---to construct a direct lift of integer points---leads to the natural question of how slack matrices and lifts of affine semigroups can be specialized to integer points in polytopes and whether this can be formulated in such a way that it has useful implications in integer programming.  

While we have not provided a complete theory of discrete lifts, this paper establishes an initial framework on which such a theory can be built. Moreover, we believe it provides some indication that the study of such lifts is not only interesting in its own right but also might lead to developments in other areas of mathematics.

\bibliography{mybib}{}
\bibliographystyle{alpha}

\end{document}